\documentclass[11pt]{article}

\usepackage{mathptmx,eucal,amsmath,amscd,amssymb,amsthm,xspace}
\usepackage[nohug,midshaft]{diagrams}
\usepackage{hyperref,amssymb,color, url,fancyhdr}

\theoremstyle{plain}

\newtheorem{thm}{Theorem}[section]

\newtheorem{cor}[thm]{Corollary}

\newtheorem{conj}{Conjecture}

\newtheorem{pro}[thm]{Proposition}

\newtheorem{lem}[thm]{Lemma}

\theoremstyle{definition}

\newtheorem{ques}[thm]{Question}

\newtheorem{exa}[thm]{Example}
\newtheorem{rem}[thm]{Remark}

\def\mf#1{\mathfrak{#1}}
\def\mc#1{\mathcal{#1}}
\def\mb#1{\mathbb{#1}}
\def\tx#1{{\rm #1}}

\def\R{\mathbb{R}}
\def\C{\mathbb{C}}
\def\Q{\mathbb{Q}}
\def\A{\mathbb{A}}
\def\Z{\mathbb{Z}}
\def\N{\mathbb{N}}
\def\F{\mathbb{F}}

\def\ol#1{\overline{#1}}

\def\hat{\widehat}

\def\rw{\rightarrow}

\def\sm{\smallsetminus}

\newarrow{Equals}{=}{=}{=}{=}{}

\DeclareMathOperator{\GL}{GL}

\DeclareMathOperator{\lcm}{lcm}
\DeclareMathOperator{\D}{D}

\DeclareMathOperator{\DM}{F}

\newcommand{\nid}{\noindent}

\newcommand{\op}{\oplus}

\fancypagestyle{plain}

\begin{document}
\bibliographystyle{plain}

%-----------------------------------------------------------
%-----------------------------------------------------------

\title{\textbf{Quantifying residual finiteness\\ of arithmetic groups}}
\author{Khalid Bou-Rabee and Tasho Kaletha}
\maketitle

%-----------------------------------------------------------
%-----------------------------------------------------------
%---------------------Abstract---------------------------

\begin{abstract}
The normal Farb growth of a group quantifies how well-approximated the group is by its finite quotients.
We show that any $S$-arithmetic subgroup of a higher rank Chevalley group $G$ has normal Farb growth $n^{\dim(G)}$.
\end{abstract}

\smallskip\smallskip

\nid keywords: \emph{Arithmetic groups, normal Farb growth, residual finiteness}

%-----------------------------------------------------------
%-----------------------------------------------------------
\section{Introduction}

The quantification of residual finiteness, begun in \cite{B10}, seeks to describe how well a residually finite group is approximated by its finite quotients. This is measured by the normal Farb growth of the group.
During a geometry seminar at Yale University in December 2009, Daniel Mostow asked the following question:

\begin{ques} (D.~Mostow)
Does asymptotic information of residual finiteness characterize arithmetic subgroups of a given linear algebraic group?
\end{ques}

\noindent This paper presents a first major step towards answering this question, by showing that in a fixed Chevalley group $G$, all $S$-arithmetic subgroups share the same normal Farb growth, and moreover this growth is $n^{\dim(G)}$. Note that for us, a Chevalley group will be a split simple algebraic group that is not necessarily simply-connected.

To state our results more precisely, we need some notation.
Let $\Gamma$ be a finitely generated, residually finite group, and let $X$ be a finite generating set for $\Gamma$. For $\gamma \in \Gamma$, let $\| \gamma \|_X$ denote the word length of $\gamma$ with respect to $X$.
Define
$$\D_\Gamma (\gamma) := \min \{ |Q| \::\: \text{ $Q$ is a finite quotient of $\Gamma$ where $\gamma \neq 1$}\},$$
and
$$\DM_{\Gamma,X}(n) := \max\{  \D_\Gamma (\gamma) \;:\;\| \gamma \|_X \leq n \}.$$
The function $\DM_{\Gamma,X}$ is called the \emph{normal Farb growth function}. It is known that the asymptotic behavior of $\DM_{\Gamma,X}$ is independent of $X$ (see Section \ref{PreliminariesSection}). The asymptotic growth of this function is called the \emph{normal Farb growth} of $\Gamma$.

The main results of this paper characterize the normal Farb growth of $S$-arithmetic groups in Chevalley groups. We use the term $S$-arithmetic subgroup of $G$ to denote any subgroup of $G(\C)$ which is commensurable with $G(\mc{O}_{K,f})$, where $K \subset \C$ is a number field, $\mc{O}_K$ is its ring of integers, and $f \in \mc{O}_K \sm \{0\}$. That is, it is an $S$-arithmetic subgroup of $G$ in the usual definition for some number field $K$ and some finite set $S$ of places of $K$ which contains the archimedean ones, but we allow $K$ and $S$ to vary.
The ingredients used include the structure theory of split semi-simple group schemes, results on the congruence subgroup problem, Moy-Prasad filtrations, Selberg's Lemma, the prime number theorem, and the Cebotar\"ev density theorem. Furthermore, we use in an essential way the results of Lubotzky-Mozes-Raghunathan \cite{LMR01}.
%Here we use the following obvious bijection:
%If S is a finite set of places of K containing the infinite ones, then for each finite prime P in S there is some power n
%s.t. P^n=(f_P) is principal. Take f to be the product of f_P for P in S finite. Then O_K(S)=O_{K,f}, because the primes of
%K containing f are precisely the finite primes in S (unique factorization in Dedekind domains). Conversely, if f is given
%we take S to be the set of primes occurring in the decomposition of the ideal (f), together with the infinite primes.
%This provides a bijection between the sets of subrings of K given by O_K(S) and O_{K,f}.

\begin{thm} \label{MainLowerBoundTheorem}
Let $G$ be a Chevalley group of rank at least $2$, $K$ be a number field, and $f \in \mc{O}_K \sm \{ 0 \}$. If $\Gamma$ is a finitely generated subgroup of $G(\C)$ with the property that $\Gamma \cap G(\mc{O}_{K,f})$ is of finite-index in $G(\mc{O}_{K,f})$, then its normal Farb growth is bounded below by $n^{\dim(G)}$.
\end{thm}

\noindent
It is interesting to ask whether an analogous result holds in rank 1. So far, the normal Farb growth of a nonabelian free group has been bounded below by $n^{2/3}$ (see \cite{KM10}).

\begin{thm} \label{MainUpperBoundTheorem}
Let $G$ be a Chevalley group, $K$ be a number field, and $f \in \mc{O}_K \sm \{ 0 \}$. If $\Gamma$ is a finitely generated subgroup of $G(\C)$ with the property that $\Gamma \cap G(\mc{O}_{K,f})$ is of finite-index in $\Gamma$, then its normal Farb growth is bounded above by $n^{\dim(G)}$.
\end{thm}

As a corollary of Theorems \ref{MainLowerBoundTheorem} and \ref{MainUpperBoundTheorem} we have the following result.

\begin{cor}
Let $G$ be a Chevalley group of rank at least $2$. Then the normal Farb growth of every $S$-arithmetic subgroup of $G$ is precisely $n^{\dim(G)}$.
\end{cor}

\noindent
This result is surprising since in general, if $\Delta$ has finite-index in $\Gamma$, we cannot hope for $F_\Gamma \approx F_\Delta$ (see Example \ref{ex:example} at the end of Section \ref{PreliminariesSection}).
Instead, the most general result in this direction is $F_\Gamma(n) \preceq (F_\Delta(n))^{[\Gamma:\Delta]}$ (see  \cite[Lemma 1.3]{B10}).

\paragraph{Acknowledgements.}
It is our pleasure to thank Daniel Mostow for suggesting this pursuit.
We are grateful to Martin Kassabov, Alexander Lubotzky, and Gopal Prasad for helpful mathematical conversations, and wish to thank Alexander Premet for providing us with a useful reference.
Furtherú, we are thankful to Benson Farb for giving us comments on an earlier draft.

\section{Preliminaries} \label{PreliminariesSection}

Let $\Gamma$ be a finitely generated, residually finite group. For $\gamma \in \Gamma \sm \{1\}$ we define
$Q(\gamma,\Gamma)$ to be the set of finite quotients of $\Gamma$ in which the image of $\gamma$ is non-trivial. We say that these quotients detect $\gamma$. Since $\Gamma$ is residually finite, this set is non-empty, and thus the natural number
\[ \D_\Gamma(\gamma) := \min\{ |Q| : Q \in Q(\gamma,G)\} \]
is defined and positive for each $\gamma \in \Gamma \sm \{1\}$. For a fixed finite generating set $X \subset \Gamma$ we define
\[ \DM_{\Gamma,X}(n) := \max\{ \D_\Gamma(\gamma) \;:\;\gamma \in \Gamma, \| \gamma \|_X \leq n, \gamma \neq 1 \}.\]
For two functions $f,g : \N \rw \N$ we write $f \preceq g$ if there exists a natural number $M$ such that $f(n) \leq Mg(Mn)$, and we write $f \approx g$ if $f \preceq g$ and $g \preceq f$. We will also write $f \succeq g$ for $g \preceq f$ and in the case when $f \approx g$ does not hold we write $f \not \approx g$.

It was shown in \cite{B10} that if $X,Y$ are two finite generating sets for the residually finite group $\Gamma$, then $\DM_{\Gamma,X} \approx \DM_{\Gamma,Y}$. Since we will only be interested in asymptotic behavior, we let $\DM_\Gamma$ be the equivalence class (with respect to $\approx$) of the functions $\DM_{\Gamma,X}$ for all possible finite generating sets $X$ of $\Gamma$. Sometimes, by abuse of notation, $\DM_\Gamma$ will stand for some particular representative of this equivalence class, constructed with respect to a convenient generating set.

We will need to use the following auxiliary function in our proofs. For any natural number $k$, we define
\[\D_\Gamma^k (\gamma) := D_\Gamma(\gamma^k) \;\;\text{and}\;\;
\DM^k_{\Gamma,X}(n) := \max\{  \D_\Gamma^k (\gamma) \;:\;\gamma \in \Gamma, \| \gamma \|_X \leq n, \gamma^k \neq 1 \}.\]
\noindent
The next lemma, which is a consequence of Selberg's Lemma (see \cite{alperin}), reveals the potential utility of $\DM_{\Gamma, X}^k$.
\begin{lem} \label{SelbergLemma}
Let $\Gamma$ be an infinite linear group generated by a finite set $X$ and let $k$ a natural number.
Then $\DM_{\Gamma,X} \approx \DM_{\Gamma,X}^k$.
\end{lem}

\begin{proof}
The inequality $\DM_{\Gamma,X}^k(n) \leq \DM_{\Gamma,X}(kn)$ is straightforward.
It suffices to prove $\DM_{\Gamma,X}(n) \leq \DM_{\Gamma,X}^k(n)$ for all but finitely many $n$.
Let $\gamma_n$ be an element such that $D_\Gamma(\gamma_n) = \DM_{\Gamma,X}(n)$ and $\|\gamma_n\|_X \leq n$.
If $\gamma_n^k \neq 1$, then $D_\Gamma(\gamma_n) \leq D_\Gamma^k( \gamma_n)$, giving $\DM_{\Gamma,X}(n) \leq \DM_{\Gamma,X}^k(n)$. The proof will be complete if we show that $\gamma_n^k=1$ holds for only finitely many $n$.
Suppose otherwise, then by Selberg's Lemma, there exists a finite-index normal subgroup $\Delta$ of $\Gamma$ that is torsion-free, and in particular $\gamma_n \notin \Delta$ for infinitely many $n$. Since $\DM_{\Gamma,X}(n)$ is non-decreasing in $n$, it must be bounded by $[\Gamma:\Delta]$, but this contradicts the infinitude of $\Gamma$.
\end{proof}

\begin{cor} If $\Gamma$ is an infinite linear group and $X,Y$ are finite generating sets for $\Gamma$, then $\DM^k_{\Gamma,X} \approx \DM^k_{\Gamma,Y}$. \end{cor}

As with the function $\DM$, we will denote the asymptotic equivalence class of $\DM^k_{\Gamma,X}$ as $X$ varies by $\DM^k_\Gamma$.
The following example shows that the linearity assumption cannot be dropped from Lemma \ref{SelbergLemma}.

\begin{exa} \label{ex:example1}
Let $\Gamma$ be the Lamplighter group $\Z/2\Z \wr \Z$. Set $\Delta = \op_{i \in \Z} \Z/2\Z$ to be the base group of $\Gamma$ so $\Gamma/\Delta \cong \Z$.
It is easy to see that for any generating set $X$ of $\Gamma$, we have
$\DM_{\Gamma,X}^2(n) \approx \DM_\Z(n)$. Thus $\DM_{\Gamma,X}^2(n) \approx \log(n)$ by \cite[Corollary 2.3]{B10}.
We now prove that $F_\Gamma(n) \succeq (\log(n))^2$, so in particular $\DM_\Gamma \not \approx \DM_{\Gamma,X}^2$.
\end{exa}

\begin{proof}
Let $\delta_i \in \Delta$ be the element given by the $i$th Kronecker delta function.
For $k$ a natural number greater than $4$, set $\gamma_k := \delta_1 + \delta_{\lcm(1, \ldots, k)}$.
Let $\phi : \Gamma \to P$ be a homomorphism to a finite quotient of $\Gamma$ that realizes $\D_\Gamma(\gamma_k)$.
We first claim that if $\delta_1 + \delta_{1+n} \in \ker \phi$ for $n \in \N$, then $n \geq k$.
Indeed, a simple calculation shows that $\delta_1 + \delta_{1+mn} \in \ker \phi$ for any $m \in \N$.
If $n \leq k$, we have that $\lcm(1, \ldots, k)$ is a multiple of $n$, so
$\delta_1 + \delta_{\lcm(1, \ldots, k)} \in \ker \phi,$
which is impossible.

Next, we claim that the set $S := \{ (\delta_n, t) : n, t \in \{ 1, \ldots, \lfloor k/4 \rfloor \} \} \subseteq \Gamma$ injects into $P$ through $\phi$.
Suppose not, then $(\delta_n, t) (\delta_{n'}, t')^{-1} \in \ker \phi$
for $t,t', n,n' \in \{ 1, \ldots ,  \lfloor k/4 \rfloor \}$ with $(\delta_n, t) \neq (\delta_{n'}, t')$.
Set $\alpha = (\delta_n, t) (\delta_{n'}, t')^{-1} = (\delta_n + \delta_{n'+t-t'}, t - t')$.
If $t - t' = 0$, then by our first claim $n = n'$ or $| |n| - |n'| | \geq k$.
If $n = n'$, then $\alpha = (0,0)$, while the latter possibility contradicts $| |n| - |n'| | \leq k/2$.
If $t - t' \neq 0$, because $\alpha \delta_i \alpha^{-1} \delta_i^{-1} \in \ker \phi$ for all $i$, we have
$\delta_{1+t-t'} + \delta_1 \in \ker \phi,$ where by our first claim, $|t - t'| \geq k$, however $|t - t'| \leq k/2$.
Our second claim is now shown.

Since $S$ injects into $P$, we have $|P| \geq \lfloor k/4 \rfloor^2$.
Fix a finite generating set $X$ for $\Gamma$, by the prime number theorem, there exists a natural number $M$ such that $\|\gamma_k \|_X \leq M 3^k$.
Set $k = \lfloor \log_3(n) \rfloor$, then because $F_\Gamma$ is increasing we have, for sufficiently large $n$,
\[
F_\Gamma(Mn) \geq F_\Gamma(M 3^k) \geq F_\Gamma(\| \gamma_k \|_X) \geq \lfloor k/4 \rfloor^2
\geq \frac{1}{32} \left[ \frac{\log(n)}{\log(3)} \right]^2.
\]
\end{proof}

\begin{lem}\label{DivisibilityAsymptoticLemma}
Let $\Gamma,\Delta$ be finitely generated and residually finite. Then \begin{itemize} \item If $\Delta \subset \Gamma$, then $\DM_\Delta \preceq \DM_\Gamma$.
\item If $f:\Delta \rw \Gamma$ is surjective with finite kernel, then $\DM_{\Delta} \preceq \DM_\Gamma$. If moreover $\tx{ker}(f)$ is central in $\Delta$ and $\Gamma$ is linear, then $\DM_\Delta \approx \DM_\Gamma$.
\end{itemize}
\end{lem}
\begin{proof}
The first assertion is \cite[Lemma 1.1]{B10}. Consider the second assertion. The inequality $\DM_{\Delta} \preceq \DM_\Gamma$ is straightforward. Assuming now that $\tx{ker}(f)$ is central in $\Delta$, we will show $\DM^k_\Delta \succeq \DM_\Gamma$, where $k=|\tx{ker}(f)|$. To that end, fix a finite generating set $X$ for $\Delta$ and use its image for $\Gamma$. Construct $\DM_\Delta$ and $\DM_\Gamma$ with respect to these generating sets. Let $g \in \Delta$, $g^k\neq 1$. Since $g^k=(zg)^k$ for all $z \in \tx{ker}(f)$, we see that $\tx{ker}(f)N$ is a normal subgroup of $\Delta$ not containing $g$. Thus $\D^k_\Delta(g) \geq \D_\Gamma(f(g))$ for all $g \in \Delta$ with $g^k \neq 1$. We now need to handle torsion elements in $\Gamma$.

For each natural number $n$, let $\gamma_n \in \Gamma$ be an element satisfying $\D_\Gamma(\gamma_n) = \DM_\Gamma(n)$ and $\| \gamma_n \| \leq n$.
Since $f$ is surjective and by our choice of generating sets, there exists $g_n \in \Delta$ such that $f(g_n) = \gamma_n$ and $\| g_n \| \leq n$.
Then if $g_n^k = 1$ for infinitely many $n$, then $\gamma_n^k =1 $ for infinitely many $n$. Following the Selberg Lemma application from Lemma \ref{SelbergLemma}, we see that $\Gamma$ is finite, which is impossible.
Thus, $g_n^k \neq 1$ for all but finitely many $n$.
For such $n$, we have $\D_\Delta^k (g_n) \geq \D_\Gamma (f(g_n))$ and hence $\DM_\Delta^k(n) \succeq \DM_\Gamma(n)$.
\end{proof}

We finish the preliminaries section with an example that illustrates that normal Farb growth of a group may be different from that of a finite index subgroup.

\begin{exa} \label{ex:example}
Let $Q$ be the subgroup of $\GL_2(\Z)$ generated by
\[
A=\begin{pmatrix} 1 & 0 \\ 0 & -1 \end{pmatrix} \qquad
\text{ and } \qquad
B = \begin{pmatrix} 0 & 1 \\ 1 & 0 \end{pmatrix}.
\]
Let $\Delta = \Z \times \Z$ and set $\Gamma = \Delta \rtimes Q$, where $Q$ acts on $\Delta$ via the standard action of $\GL_2(\Z)$.
Because $Q$ is finite, $\Gamma$ contains $\Delta$ as a subgroup of finite-index.
Further, $F_{\Delta}(n) \approx \log(n)$ by \cite[Corollary 2.3]{B10}.
We now prove that $F_\Gamma(n) \succeq (\log(n))^2.$
\end{exa}

\begin{proof}
Let $X$ be a generating set for $\Gamma$ containing $(1,0)$ and $(0,1)$ in $\Delta$.
Set $\gamma_k$ to be $(\lcm(1, \ldots, k), 0) \in \Delta$.
By the prime number theorem, there exists a natural number $M$ such that $\|\gamma_k\|_X \leq M 3^k$.
Let $\phi: \Gamma \to P$ be a homomorphism to a finite quotient of $\Gamma$ that realizes $D_\Gamma(\gamma_k)$ and set $V = \ker \phi \cap \Delta$.
We first construct a subgroup of $V$ of the form $d \Z \times d \Z$ for some natural number $d$.
Consider the intersection of $V$ with $\Z \times 0$.
This is a subgroup of $\Z$, hence is isomorphic to $d \Z$ for some natural number $d$.
Thus we have $d \Z \times 0 \subset V$, and conjugating by $B$ we also find $0 \times d \Z$ is in $V$.

Next, we claim that the index of $d \Z\times d\Z$ in $V$ is at most 4: Let $(a,b) \in V$.
Then $(2a,0)=(a,b)+A(a,b)A^{-1} \in V$, and similarly $(2b,0) \in V$, so $2a,2b \in d\Z$, and hence $2(a,b) \in d\Z \times d\Z$, which shows that every element of $V/d\Z \times d\Z$ has order (at most) 2. But $V$ is a free abelian group of rank 2, so $V/d\Z \times d\Z$ is generated by two elements, and the claim follows.
We conclude that $d^2 = [\Delta: d \Z \times d \Z] = [\Delta: V][V: d \Z \times d \Z] \leq 4 [\Delta:V]$, giving $|P| \geq \frac14 d^2$.

Finally, since $\gamma_k \notin \ker(\phi)$, we must have that $d \geq k$.
Hence, $F_\Gamma(M 3^k) \geq D_\Gamma(\gamma_k) \geq \frac14 k^2$.
Set $k = \lfloor (\log_3(n)) \rfloor$, then because $F_\Gamma$ is increasing we have, for sufficiently large $n$,
\[
F_\Gamma(M n) \geq F_\Gamma( M 3^k ) \geq \frac{1}4 k^2 \geq \frac{1}{16} \left( \frac{\log(n)}{\log(3)} \right)^2,
\]
giving $F_\Gamma(n) \succeq (\log(n))^2,$ as desired.
\end{proof}

%=================================================================================================
%-------------------------------------------------------------------------------------------------
%=================================================================================================

\section{Lower bounds} \label{LowerBounds}

Let $G$ be a Chevalley group, i.e. a split simple group scheme defined over $\Z$, and let $\mf{g}$ be its Lie-algebra. Note that we do not assume that $G$ is simply-connected. For a natural number $m$, we put $G(m) = G(\Z/m\Z)$. For a while, we will focus attention on the powers of a single prime $p$, and to lighten the notation we put $G_k = G(\Z/p^k\Z)$.

Recall from \cite[exp.1, 2.3.3+2.3.6]{SGA3} the definition of the center $Z(G)$ of $G$. It is the subfunctor of $G$, which assigns to each scheme $S$ the following subgroup of $G(S)$
\[ Z(G)(S) := \left\{ g \in G(S)|\ \forall S' \rw S:\tx{Ad}(g)|_{G(S')} = \tx{id}_{G(S')}  \right\} \]
where $\tx{Ad}(g)|_{G(S')}$ denotes the automorphism of $G(S')$ provided by conjugation by the image of $g$ under the natural map $G(S) \rw G(S')$.
\pagebreak[1]

It is shown in \cite[exp.22, 4.1.8]{SGA3} that the functor $Z(G)$ is representable by a closed $\Z$-subgroup-scheme of $G$, which is finite and diagonalizable. As such, $Z(G)$ is a product of finitely many groups schemes, each isomorphic to $\mu_n$ for some $n$, where $\mu_n$ is the group scheme of $n$-th roots of unity. In particular, $Z(G)$ is etale over $\Z[\tx{ord}(Z(G))^{-1}]$. See \cite[exp.8, 2.1]{SGA3}.

From the definition it is obvious that $Z(G)(S) \subset Z(G(S))$. We will show that there exists $f \in \Z \sm \{0\}$ such that if $S$ lies over $\tx{Spec}(\Z_f)$, then $Z(G)(S)=Z(G(S))$. The main ingredient in this proof is the following lemma, which asserts the existence of a strongly regular section of the split maximal torus in $G$ over $\tx{Spec}(\Z_f)$.

\begin{lem} Let $T \subset G$ be a split maximal torus. There exists $f \in \Z\sm\{0\}$ and a point $s \in T(\Z_f)$ such that
\[ \tx{Cent}(s,G \times \tx{Spec}(\Z_f)) = T \times \tx{Spec}(\Z_f). \]
\end{lem}
\begin{proof}
Consider the closed subscheme of $T$ given by
\[ \bigcup_{\alpha \in R(T,G)}\tx{ker}(\alpha) \cup  \bigcup_{w \in W} T^w \]
where $R(T,G)$ is the set of roots of $T$ in $G$ and $W=\tx{Norm}(G,T)/T$ is the Weyl group. Let $U$ be its complement in $T$. Then $U \rw T$ is an open immersion, which when composed with an isomorphism $T \cong \mb{G}_m^r$ and the open immersion $\mb{G}_m^r \rw \A_\Z^r$ provides an open immersion $U \rw \A_\Z^r$. Since $\A^r(\Q)$ is dense in $\A^r(\ol{\Q})$, it follows that $U(\Q) \neq \emptyset$. As $U$ is of finite type, any map $\tx{Spec}(\Q) \rw U$ factors as $\tx{Spec}(\Q) \rw \tx{Spec}(\Z_f) \rw U$ for some $f$. Thus we have a point $s : \tx{Spec}(\Z_f) \rw U$. We claim that this point satisfies the statement of the lemma. To lighten notation, let us base change to $\tx{Spec}(\Z_f)$. Consider the centralizer $H := \tx{Cent}(s,G)$. It is a closed subscheme of $G$, hence affine and of finite type over $\Z_f$, and contains $T$. By generic flatness, we may assume that $H$ is flat, after possibly changing $f$. By the choice of $s$, all fibers of $H$ and $T$ coincide. By \cite[exp. 10, 4.9]{SGA3}, $H$ is a torus, and since $T$ is a maximal torus, it follows that $H=T$.
\end{proof}

\begin{cor} \label{cor:center} There exists $f \in \Z\sm\{0\}$ such that for all schemes $S \rw \tx{Spec}(\Z_f)$ we have
\[ Z(G)(S) = Z(G(S)).\]
\end{cor}
\begin{proof} The inclusion $\subset$ is obvious from the definition of $Z(G)$ and we now have to show the converse. Choose $f$ and $s \in T(\Z_f)$ as in the above lemma. Let $S \rw \tx{Spec}(\Z_f)$ and $x \in Z(G(S))$. If $s_S \in T(S)$ denotes the image of $s$ under $T(\Z_f) \rw T(S)$, then
\[ x \in \tx{Cent}(s_S,G_S)(S) = T(S)\]
We claim that for every root $\alpha \in R(T_S,G_S)$ we have $\alpha(x)=1$. Assume by way of contradiction that this were not the case. Let $u_\alpha : \mb{G}_{a,S} \rw G_S$ be the root subgroup corresponding to $\alpha$, and $y = u_\alpha(1)$. Then $y \in G(S)$ is a point not centralized by $x$, contrary to the assumptions. It follows that
\[ x \in \bigcap_{\alpha \in R(T_S,G_S)} \ker(\alpha)(S) = Z(G)(S)\]
where the last equality is \cite[exp. 22, 4.1.6]{SGA3}.
\end{proof}

\begin{cor} \label{cor:centerord} There exists a finite set of primes $P$ such that $|Z(G_k)|$ divides $\tx{ord}(Z(G))$ for all primes $p \notin P$. In particular, if $m$ is an integer coprime to the elements of $P$, then the order of every element of $Z(G(m))$ divides $\tx{ord}(Z(G))$.
\end{cor}
\begin{proof} The second statement is an immediate consequence of the first, since $Z(G(m)) = \prod_{p^k\|m}Z(G_k)$. To prove the first, let $P$ be the set of primes $p$ for which $Z(G)(\Z/p^k\Z)$ is a proper subgroup of $Z(G_k)$. According to Corollary \ref{cor:center} the set $P$ is finite. For a prime $p$ not in $P$, we then have $Z(G_k) = Z(G)(\Z/p^k\Z)$. As already remarked, $Z(G)$ is a finite product of $\mu_n$'s. Since $(\Z/p^k\Z)^\times$ is cyclic, the number $|\mu_n(\Z/p^k\Z)|$ divides $n$. The statements now follows.
\end{proof}

\begin{lem} \label{lem:red} The natural projection $\Z/p^k\Z \rw \Z/p^{k-1}\Z$ induces a surjective homomorphism
\[ G_k \rw G_{k-1} \]
For all but finitely many primes $p$, this homomorphism restricts to an isomorphism
\[ Z(G_k) \rw Z(G_{k-1}). \]
\end{lem}
\begin{proof}
The first claim follows directly from the infinitesimal lifting property of smoothness. For the second claim, let $p$ be a prime which does not divide $\tx{ord}(Z(G))$ and for which $Z(G_k)=Z(G)(\Z/p^k\Z)$ for all $k$. By Corollaries \ref{cor:center} and \ref{cor:centerord} these are all but finitely many primes. Then $Z(G)$ is etale over $\Z_{(p)}$ and this implies the bijectivity of the second map.
\end{proof}

\begin{cor} \label{cor:centerless} Assume that $G$ is simply-connected. Then for all but finitely many $p$,
\[ Z(G_k/Z(G_k))=\{1\}. \] \end{cor}
\begin{proof}
We prove this by induction on $k$. The base case is $k=1$, which is known, since $G(\F_p)/Z(G(\F_p))$ is simple.
%! Check the simplicity of Chevalley groups at lower primes
For the induction step, let $k>1$. Let $z \in G_k$ be an element which is central in $G_k/Z(G_k)$. Then for all $g \in G_k$, $z_g := gzg^{-1}z^{-1} \in Z(G_k)$. Under the surjection $G_k \rw G_{k-1}$, the element $z$ maps to an element $\bar z$ with the same property. Applying the induction hypothesis we see that $\bar z \in Z(G_{k-1})$. This implies, that $\bar z_g = 1$. Lemma \ref{lem:red} now implies $z_g=1$ and the statement follows.
\end{proof}

\noindent For $0 \leq i \leq k$, let $G_k^i := \tx{ker}(G_k \rw G_i)$. This provides a descending filtration
\[ G_k=G_k^0 \geq G^1_k \geq ... \geq G_k^k=\{1\}. \]

\noindent We fix a closed embedding $G \rw \tx{SL}_m$ defined over $\Z$.
% \cite[exp. 6B, 12.3]{SGA3}
This yields an embedding of Lie-algebras $\mf{g} \rw \tx{sl}_m$ defined over $\Z$. We identify $G$ and $\mf{g}$ with their respective images.
Clearly $G_k^i = [1+p^iM_m(\Z/p^k\Z)] \cap G_k$, and an element $1+p^ix \in G_k^i$ belongs to $G_k^{i+1}$ if and only if $x \equiv 0 \tx{\ mod\ } p$.

The following Lemma is a well-known result from the theory of Moy-Prasad filtrations \cite{MP94}.
\begin{lem}[Moy-Prasad] \label{lem:filq}\ \\[-20pt]
\begin{enumerate}
\item $[G_k^i,G_k^j] \subset G_k^{i+j}$.
\item For $1 \leq i \leq k-1$ the map
\[ G_k^i/G_k^{i+1} \rw \mf{g}(\F_p), \qquad 1+p^ix \mapsto x \tx{\ mod\ }p \]
induces an isomorphism of groups, which is equivariant with respect to the action of $G(\F_p)$ on both sides by conjugation.
\end{enumerate}
\end{lem}

\begin{rem} In particular, one sees inductively that each $G_k^i$ for $i>0$ is a $p$-group. \end{rem}

\begin{lem} \label{lem:dimfin} There exists positive constants $c,C$ such that for all prime powers $m=p^k$
\[ cm^{\dim(G)} \leq |G(m)| \leq Cm^{\dim(G)}. \]
\end{lem}
\begin{proof} In the case $k=1$ the lemma follows from
\cite[Theorem 25, Section 9]{S68}. The general case reduces to this, because according to Lemma \ref{lem:filq} we have
$|G_k|=p^{(k-1)\dim(G)}|G(\F_p)|$.
\end{proof}

\begin{lem} \label{lem:noinv}
For all but finitely many $p$, the Lie-algebra $\mf{g}(\F_p)$ has no center, and the adjoint action of $G(\F_p)/Z(G(\F_p))$ on $\mf{g}(\F_p)$ is faithful and irreducible.
\end{lem}
\begin{proof} This is a well-known classical result. See for example \cite{H84}, \cite{H82}. \end{proof}

\begin{lem} \label{lem:annuli} Assume that $p$ is sufficiently large, and $0 \leq i \leq k-2$. For every $g \in G_k^i \sm G_k^{i+1}Z(G_k)$ there exists $h \in G_k^1$ such that $hgh^{-1}g^{-1} \in G_k^{i+1} \sm G_k^{i+2}Z(G_k)$.
\end{lem}

\begin{proof}
Note first that by Lemma \ref{lem:red}, $G_k^{i+1} \cap (G_k^{i+2}Z(G_k)) = G_k^{i+2}$. Hence it is enough to find $h$ such that $hgh^{-1}g^{-1} \notin G_k^{i+2}$.

Write $h=1+py$ with some $y \in M_m(\Z/p^k\Z)$ to be determined. We will make use of the following computation: For any $x \in M_m(\Z/p^k\Z)$ we have
\begin{eqnarray*}
(1+py)x(1+py)^{-1}&=&(x+pyx)(1+py)^{-1}\\
&=&(x+pxy-p[x,y])(1+py)^{-1}\\
&=&(x-p[x,y](1+py)^{-1})
\end{eqnarray*}
where $[x,y]=xy-yx$.

First assume that $i=0$. Then using the above computation we see that
\[ hgh^{-1}g^{-1} = 1-p[g,y](1+py)^{-1}g^{-1} \]
Clearly the right hand side belongs to $G_k^1$, and to show that it does not belong to $G_k^2$ it is enough by Lemma \ref{lem:filq} to show that the reduction mod $p$ of the matrix $[g,y](1+py)^{-1}g^{-1} \in M_m(\Z/p^k\Z)$ is non-zero. Call this reduction $T$. It belongs to $\mf{g}(\F_p)$. Using the formula
\[ (1+py)^{-1} = \sum_{j=0}^{k-1} (-py)^j \]
we compute that $T=[\bar g,\bar y]\bar g^{-1} = \bar g\bar y\bar g^{-1}-\bar y$. By Lemma \ref{lem:red}, the preimage of $Z(G(\F_p))$ under $G_k \rw G(\F_p)$ is $G_k^1Z(G_k)$. Thus by assumption, the image $\bar g$ of $g$ in $G(\F_p)/Z(G(\F_p))$ is non-trivial, and by Lemma \ref{lem:noinv} there exists $\bar y \in \mf{g}(\F_p)$ such that $\bar g\bar y\bar g^{-1} \neq \bar y$. According to Lemma \ref{lem:filq}, there exists $h=1+py \in G^1_k$ corresponding to this $\bar y$. This completes the proof in the case $i=0$.

Now assume $i>0$. We write $g=1+p^ix$ for some $x \in M_m(\Z/p^k\Z)$ whose reduction mod $p$ belongs to $\mf{g}(\F_p)$. Then
\begin{eqnarray*}
&&(1+py)(1+p^ix)(1+py)^{-1}(1+p^ix)^{-1}\\
&=&(1+p^i(1+py)x(1+py)^{-1})(1+p^ix)^{-1} \\
&=&(1+p^ix-p^{i+1}[x,y](1+py)^{-1})(1+p^ix)^{-1}\\
&=&1- p^{i+1}[x,y](1+py)^{-1}(1+p^ix)^{-1}
\end{eqnarray*}
Again $hgh^{-1}g^{-1} \in G_k^{i+1}$, and we want to choose y so that this element does not belong to $G_k^{i+2}$. By Lemma \ref{lem:filq}, this is equivalent to the demand that the reduction mod $p$ of the element
\[ [x,y](1+py)^{-1}(1+p^ix)^{-1} \in M_m(\Z/p^k\Z) \]
be non-trivial. Using the formula for $(1+p^ix)^{-1}$ analogous to that used above for $(1+py)^{-1}$ we compute that this element is equal mod $p$ to $[x,y]$. Now we consider the image of $[x,y] \in M_m(\F_p)$. Of course, this is just the bracket of the images of $x$ and $y$ in $M_m(\F_p)$. But these images, and hence their bracket, lie in $\mf{g}(\F_p)$. Again as in the case $i=0$, specifying $h$ is equivalent to choosing the class of $y$ in $\mf{g}(\F_p)$ in such a way that its bracket with the class of $x$ is non-trivial. Since the Lie-algebra $\mf{g}(\F_p)$ has no center, the class of $x$ is non-central, and so an appropriate $y$ exists.
\end{proof}

\begin{pro} \label{ContainingCenterPro}
Assume that $p$ is sufficiently large and $G$ is simply-connected. Then every normal subgroup $N < G_k$ which contains $Z(G_k)$ equals $G_k^iZ(G_k)$ for some $i$.
\end{pro}
\begin{proof}
We will first prove under the assumption $k>1$ by descending induction on $i$ the following statement.
\[ \forall 0 \leq i < k:\quad N\cap [G_k^i \sm G_k^{i+1}Z(G_k)] \neq \emptyset\quad \Rightarrow\quad G_k^i \subset N \]
The base case is when $i=k-1>0$. Then the isomorphism of Lemma \ref{lem:filq} identifies $G_k^i$ with $\mf{g}(\F_p)$ and $N \cap G_k^i$ with an invariant subspace of $\mf{g}(\F_p)$. By assumption this space is non-trivial, and by Lemma \ref{lem:noinv} it is all of $\mf{g}(\F_p)$, hence $N \cap G_k^i = G_k^i$. For the induction step, assume $i \geq 0$. Let $g \in N \cap [G_k^i \sm G_k^{i+1}Z(G_k)]$. Use Lemma \ref{lem:annuli} to obtain $h \in G_k^1$ such that $hgh^{-1}g^{-1} \in G_k^{i+1} \sm G_k^{i+2}Z(G_k)$. Then $hgh^{-1}g^{-1} \in N$, and we may apply the induction hypothesis to conclude $G_k^{i+1} \subset N$. Now look at the normal subgroup $(N \cap G_k^i)/G_k^{i+1}$ of $G_k^i/G_k^{i+1}$.
If $i>0$, then we have the isomorphism $G_k^i/G_k^{i+1} \rw \mf{g}(\F_p)$ and the image of that normal subgroup is a non-trivial invariant subspace. If $i=0$, then we have the isomorphism $G_k^i/G_k^{i+1} \rw G(\F_p)$ and the image of that normal subgroup is normal subgroup of $G(\F_p)$ which properly contains $Z(G(\F_p))$. In both cases, we conclude that $(N \cap G_k^i)/G_k^{i+1} = G_k^i/G_k^{i+1}$, and hence $N \cap G_k^i = G_k^i$. This completes the induction.

Now we show how the proposition follows from the above statement. The case $k=1$ is trivial since $G_1/Z(G_1)$ is simple. Thus assume $k>1$. If $N=Z(G_k)$ there is nothing to prove. Otherwise there exists a unique smallest index $i$ such that $G_k^i \sm G_k^{i+1}Z(G_k)$ contains an element of $N$. By the above statement, $Z(G_k)G_k^i \subset N$, but by minimality of $i$ this must in fact be an equality.
\end{proof}

\begin{pro} \label{PCPOntoPro}
Let $N$ be a natural number, and $H = \tx{ker}[G(\Z) \rw G(N)]$. If $G$ is simply-connected, then for any $m$ coprime to $N$ the projection $G(\Z) \rw G(m)$ maps $H$ surjectively onto $G(m)$.
\end{pro}

\begin{proof}
We begin with the special case $N=1$, then $H=G(\Z)$.
Since $G$ is smooth, the natural projection $G(\Z_p) \rw G(\Z/p^k\Z)$ is
surjective for all primes $p$ and all natural numbers $k$, and hence the
natural projection $G(\hat \Z) \rw G(m)$ is surjective for all natural
numbers $m$. By strong approximation (\cite{PR94}), the inclusion $G(\Z) \rw
G(\hat \Z)$ has dense image. Thus, the natural projection $G(\Z) \rw G(m)$
is surjective.

For the general case, we have $G(Nm) \cong G(N) \times G(m)$, and by the
first part of the proof, the projection $G(\Z) \rw G(N) \times G(m)$ is
surjective. The preimage in $G(\Z)$ of the subgroup $1 \times G(m)$ of $G(N)
\times G(m)$ is precisely $H$, and maps surjectively onto $G(m)$.
\end{proof}

\begin{pro} \label{pro:unipmetric} Assume that the rank of $G$ is at least $2$. Let $u : \mb{G}_a \rw G$ be a root subgroup, and $X$ a finite generating set for $G(\Z)$. Then there exists a positive constant $M$ such that for any positive $z \in \Z$
\[ \|u(z)\|_X \leq M\log(z). \]
\end{pro}
\begin{proof}
Composing $u$ with the chosen closed embedding $G \rw \tx{SL}_m$, and then further with the natural inclusion $\tx{SL}_m \rw \tx{M}_m$, we obtain a morphism of $\Z$-schemes
\[ u' : \A_\Z^1 \rw \A_\Z^{m^2} \]
which is given by collection $\{u'_{i,j}\}$ of $m^2$-many polynomials in one variable with integral coefficients. Let $k = \max{\deg(u'_{i,j})}+1$. Then there exists a positive constant $C$ such that $u'_{i,j}(z) \leq Cz^k $ for all positive integers $z$ and all $i,j$. Thus $\|u'(z)\| \leq Cz^k$ for all $z \in \N$, where $\|\ \|$ is the maximum norm on $\tx{M}_m(\R)$. The result now follows from Theorem A in \cite{LMR01}.
\end{proof}

We are now ready to prove our main lower bound. In the proof, we are going to use the fact that if $G$ is simply-connected and has rank at least 2, then $G(\Z)$ has the congruence subgroup property. We refer the reader to \cite[Chap. 9.5]{PR94} for a discussion of this property.
Also recall that a subgroup of $G(m)$ is called \emph{essential} if it does not contain the kernel of the natural map $G(m) \to G(r)$ for any $r |m$ with $r < m$.
\begin{thm} \label{MainLowerBoundPro}
Assume that the rank of $G$ is at least $2$. Let $K$ be a number field, $f \in \mc{O}_K$, and $\Delta$ a finitely generated subgroup of $G(\C)$ with the property that $\Delta \cap G(\mc{O}_{K,f})$ is of finite-index in $G(\mc{O}_{K,f})$.
Then
\[ \DM_\Delta(n) \succeq n^{\dim(G)}. \]
\end{thm}

\begin{proof} %It is known that $\Delta$ is finitely generated \cite[Theorem 5.11]{PR94} and residually finite, hence the function $\DM_\Delta$ is defined.
Let $G_\tx{sc}$ be the simply connected cover of $G$, and $p : G_\tx{sc}(\mc{O}_{K,f}) \rw G(\mc{O}_{K,f})$ the natural map. Then $\Delta_\tx{sc} := p^{-1}(\Delta \cap G(\mc{O}_{K,f}))$ is of finite index in $G_\tx{sc}(\mc{O}_{K,f})$ and the map $p : \Delta_\tx{sc} \rw \Delta$ has finite kernel. By Lemma \ref{DivisibilityAsymptoticLemma} we may assume for the rest of the proof that $G=G_\tx{sc}$ and $\Delta \subset G(\mc{O}_{K,f})$.

Since $\Delta$ is of finite-index in $G(\mc{O}_{K,f})$, so is $\Delta \cap G(\Z)$ of finite-index in $G(\Z)$. By virtue of the congruence subgroup property of $G(\Z)$, we can find a principal congruence subgroup $\Delta' \subset \Delta \cap G(\Z)$. Applying again Lemma \ref{DivisibilityAsymptoticLemma}, we may assume for the rest of the proof that $\mc{O}_{K,f}=\Z$ and that $\Delta$ is a principal congruence subgroup of $G(\Z)$.

Let $N = \tx{ord}(Z(G))$.
By Lemma \ref{SelbergLemma}, it suffices to find a lower bound for $F^N_\Delta$.
Loosely speaking, we will see that working with $F^N_\Delta$ instead of $F_\Delta$ will aid us in ignoring certain central elements in finite images of $\Delta$.

We first construct candidates that are poorly approximated by finite quotients.
Let $X$ and $Y$ be finite generating sets for $G(\Z)$ and $\Delta$ respectively. Let $S$ be the set of primes $p$ for which at least one of the following conditions fails

\begin{itemize}
\item $|Z(G_k)|$ divides $N$,
\item If $Z(G_k) \trianglelefteq N \trianglelefteq G_k$ then $N=G_k^iZ(G_k)$ for some $i$,
\item The projection $G(\Z) \rw G_k$ maps $\Delta$ surjectively onto $G_k$.
\end{itemize}
where as before $G_k=G(\Z/p^k\Z)$. By Corollary \ref{cor:centerord} and Propositions \ref{ContainingCenterPro} and \ref{PCPOntoPro} this set is finite.
Put $\alpha = \prod_{p \in S} p$ and $r_k = \alpha^k\lcm(1, \ldots, k)$. Let $u : \mb{G}_a \rw G$ be a root subgroup, and $B_k = u(r_k)$. Since $u$ is defined over $\Z$, we have $B_k \in G(\Z)$, hence $A_k := B_k^{[G(\Z) : \Delta]} \in \Delta$. The elements $A_k$ will be our candidates for achieving lower bounds for $F_\Delta^N$.

Next, we bound the word length of $A_k$, i.e. the function $k \mapsto \|A_k\|_Y$.
By Proposition \ref{pro:unipmetric} there exists a natural number $M$ such that
\[\|A_k\|_X  \leq M \log( \lcm(1, \ldots, k) \alpha^k).\]
Hence, by the prime number theorem we may find a potentially different natural number $M$ so that
$\|A_k\|_X \leq M k.$
Finally, since $G(\Z)$ is quasi-isometric to $\Delta$, we have that
\begin{equation} \label{CandidateBallSize}
\| A_k \|_Y \leq M k,
\end{equation}
for a some other natural number $M$.

The remainder of the proof is devoted to finding a lower bound for the cardinality of any finite quotient $Q=\Delta/H$ which detects $A_k^N$, in particular to the quotient realizing $D^N_\Delta(A_k)$. We start by taking one such quotient $Q$. Since we are looking for a lower bound of the cardinality of $Q$, we may replace it by either a subgroup or a quotient of it, and we will do so repeatedly in the following.

By the congruence subgroup property for $G(\Z)$ there exists a natural number $m$ such that the kernel of the  projection $\phi: G(\Z) \to G(m)$ lies in $H$.
Let $\Delta'$,$H'$, and $A_k'$ be the images of $\Delta$, $H$, and $A_k$ respectively in $G(m)$.
By the Chinese remainder theorem, we may write $G(m) = A \times B$ where
\[
A = \prod_{\substack{p^j \| m\\ p \in S}} G(p^j) \qquad \text{ and }\qquad B =  \prod_{\substack{p^j \| m\\ p \not\in S}} G(p^j).
\]
and $p^j\|m$ means that $j$ is the greatest power of $p$ which divides $m$.

We know $(A_k')^N \neq 1$. For any $c \in Z(B)$ we have $\tx{ord}(c)|N$ (see Corollary \ref{cor:centerord} and the choices of $S$ and $N$). Thus we have $(c A_k')^N = (A_k')^N$ for any $c \in Z(B)$, which implies $c A_k' \notin H'$.
Hence, $A_k' \notin H' Z(B)$.
Letting $A_k''$, $\Delta''$, and $H''$ be the images of $A_k'$, $\Delta'$, and $H'$ in $A \times B/Z(B)$ respectively, we have that $A_k'' \notin H''$.
Further, $[\Delta'':H''] \leq [\Delta':H']$ since $\Delta''/H''$ is an image of $Q = \Delta'/H'$.

We claim that any quotient of $B/Z(B)$ is centerless: Indeed, by the choice of $S$, for every $p \notin S$, Lemma \ref{lem:red} and Corollary \ref{cor:centerless} imply that all quotients of $G(p^j)/Z(G(p^j))$ are centerless. By \cite[1.4]{LL01} every normal subgroup of $B/Z(B)$ is a product of normal subgroups of the factors of $B/Z(B)$, and the statement follows.

Recall that $\Delta$ was assumed to be a principal congruence subgroup of $G(\Z)$. By Proposition \ref{PCPOntoPro}, $G(\Z)$ projects onto $A \times B/Z(B)$. Hence, $\Delta''$ is normal in $A \times B/Z(B)$, and applying \cite[1.3,1.4]{LL01} we see that
$\Delta'' = \Delta_1 \times \Delta_2$, where
\[
\Delta_1 = \pi_1(\Delta'') \qquad \text{ and } \qquad \Delta_2 =  \pi_2(\Delta''),
\]
where $\pi_1$ and $\pi_2$ are the natural projection maps of $A \times B/Z(B)$ onto $A$ and $B/Z(B)$ respectively.

By the choice of $S$ we have $\Delta_2=B/Z(B)$. The subgroup $H''$ is normal in $\Delta'' = \Delta_1 \times \Delta_2$, and since $\Delta_2$ has no center, \cite[Corollary 1.4]{LL01} applies again giving $H'' = H_1 \times H_2$ where $H_1 = \pi_1(H'')$ and $H_2 = \pi_2(H'')$.
Now since $A_k'' \notin H_1 \times H_2$ we have two cases: $\pi_1(A_k'') \notin \pi_1 (H'')$ or $\pi_2(A_k'') \notin \pi_2(H'').$
In both cases, we claim that there exists a natural number $M$, independent of $k$, such that
$M|Q| \geq k^d$, where $d:=\dim(G)$.

We first handle the case $\pi_1(A_k'') \notin \pi_1 (H'')$.
Write $A = G(m_0)$, let $r$ be the smallest natural number such that the kernel of the natural map $\phi: G(m_0) \to G(r)$ is contained in $\pi_1(H'')$.
Then $\phi(\pi_1(A_k'')) \notin \phi(\pi_1(H''))$ and $\phi(\pi_1(H''))$ is essential or trivial.
Since the image of $A_k$ in $G(r)$ is nontrivial, $r$ does not divide $\alpha^k$. But any prime dividing $r$ also divides $\alpha$ (recall the choices of $A$, $r$ and $\alpha$), hence $p^k|r$ for some $p \in S$.
In the case $\phi(\pi_1(H''))$ is essential, \cite[Proposition 6.1.2]{LS03} gives
$C [G(r): \phi(\pi_1(H''))] \geq r \geq p^k$, where $C$ is a natural number that only depends on $G$.
If $\phi(\pi_1(H''))$ is trivial, we get the better bound $C |G(r)| \geq C|G(p^k)|\geq p^{kd}$ by Lemma \ref{lem:dimfin} where $C$ is again a natural number that depends only on $G$.
Set $M' = C [G(\Z): \Delta]$. Since $[G(r) : \phi(\pi_1(\Delta''))] \leq [G(\Z): \Delta]$, we have
\[
M' [\Delta'' : H'' ] \geq C [G(r) : \phi(\pi_1(\Delta''))] [\phi(\pi_1(\Delta'')) : \phi(\pi_1( H''))] = C [G(r) : \phi(\pi_1( H''))] \geq p^k.
\]

There exists a natural number $M''$ such that $M'' p^k \geq k^d$ for all $p \in S$ and $k \in \N$.
Setting $M = M' M''$, we see that
\[
M [\Delta '' : H''] \geq M'' p^k \geq k^d.
\]
Since $|Q| \geq [\Delta'' : H'']$, the claim is shown.

Next we handle the case $\pi_2(A_k'') \notin \pi_2(H'')$.
By repeated use of Corollary 1.4 in \cite{LL01}, there exists a natural projection $\phi :A\times B/Z(B) \to G_k/Z(G_k)$ with $\phi(A_k'') \notin \phi(H'')$ and $G_k = G(p^k)$ where $p \notin S$.
By Proposition \ref{ContainingCenterPro} and the normality of $H_2$ in $\Delta_2 = B/Z(B)$, we have $\phi(H'') = G_k^i/Z(G_k)$ for some $i$, hence the image of $\phi(A_k'')$ through the natural projection onto $G_i/Z(G_i)$ is non-trivial. Further, $Q$ maps onto $G_i/Z(G_i)$.

From the estimate $M' |G_i| \geq p^{di}$ (Lemma \ref{lem:dimfin}), where $M'$ is a natural number, and the fact that $p^i$ does not divide $\lcm(1, \ldots, k)$ we obtain $p^i \geq k$, and thus,
$$
M' |G_i| \geq p^{id} \geq k^d.
$$
Finally, since $|G_i|/|Z(G_i)| \leq |Q|$ and $|Z(G_i)| \leq N$ (by the choice of $S$), the claim holds with $M = M'N$.

The inequality  $M |Q| \geq k^d$ in tandem with Inequality (\ref{CandidateBallSize}) gives some natural number $M$ such that
$M F_\Delta^N( k) \geq k^d,$ finishing the proof of the theorem.
\end{proof}

%=================================================================================================
%-------------------------------------------------------------------------------------------------
%=================================================================================================

\section{Upper bounds} \label{UpperBoundSection}

In this section, $G$ continues to be a Chevalley group.
Our main upper bound result is a corollary of the following three propositions.

\begin{pro} \label{OLCase}
Let $L$ be a number field with ring of integers $\mc{O}_L$. Then
\[ \DM_{\mathcal{O}_{L}}(n) \approx \log(n). \]
Moreover, the finite quotients of the form $\Z/ p\Z \cong \mathcal{O}_L/\mf{p}$, where $p$ is a prime number that splits completely in $\mathcal{O}_L$, $\mf{p} | p\mathcal{O}_{L}$, are enough to obtain the upper bound.
\end{pro}

\begin{proof} The fact $\DM_{\mathcal{O}_{L}}(n) \succeq \log(n)$ follows immediately from \cite[Thm. 2.2]{B10} and Lemma \ref{DivisibilityAsymptoticLemma}. Thus it is enough to prove $\DM_{\mathcal{O}_{L}}(n) \preceq \log(n)$.

Let $S = \{b_1, \ldots, b_k \}$ be an integral basis for $\mathcal{O}_L$, and fix a nontrivial $g$ in $\mathcal{O}_L$ with $\|g \|_S = n$.
Then $g = \sum_{i=1}^n a_i b_i$ where $a_i \in \mathbb{Z}$ and $|a_i| \leq n$.
Since $g\neq 0$ there exists  $k$ such that $a_k \neq 0$.
By the Cebotar\"ev density theorem, the set $P$ of all primes in $\mathbb{Z}$ that split in $\mathcal{O}_L$ has nonzero natural density in the set of all primes.
We claim that there exists $C > 0$, which does not depend on $n$, and a prime $q$ such that $(q)$ splits in $\mathcal{O}_L$ and $q \leq C \log(n)$ and $a_k \not\equiv 0 \mod q$.
Indeed, enumerate $P = \{ q_1, q_2, \ldots  \}$.
Let $q_{r+1}$ be the first prime in $P$ such that $a_k \not \equiv 0 \mod q_{r+1}$.
Then $q_1 \cdots q_{r}$ divides $a_k$ and by the prime number theorem and positive density of $P$, we have that $q_{r+1} \leq M r \log(r)$ for some $M > 0$, depending only on $L$.
A similar calculation shows that there exists $M'>0$ such that $q_1 \cdots q_r \geq e^{M' r\log(r)}$.
Hence, $q_{r+1} \leq C \log(a_k)$, where $C > 0$ depends only on $L$. The claim is shown.

Write $(q) = \mathfrak{q}_1 \cdots \mathfrak{q}_{c}$ with
$|\mathcal{O}_L/\mathfrak{q}_i| = q.$
Since $q$ does not divide $a_k$ and since the integral basis $S$ gets sent to a $\F_q$-basis of $\mathcal{O}_L/(q)$, we have that $g \neq 1$ in $\mathcal{O}_L/(q)$.
Hence, there exists one $\mathfrak{q}_i$ with $g \neq 1$ in $\mathcal{O}_L/\mathfrak{q}_i$.
As the cardinality of $\mathcal{O}_L/\mathfrak{q}_i$ is equal to $q$ which is no greater than $C \log(n)$, we have the desired upper bound.
\end{proof}

\begin{pro}
\label{thecaseofO_L[1/f]}
Let $\Gamma$ be a finitely generated subgroup of $G(\mc{O}_{L,f})$, where $L$ is a number field and $f \in \Z$.
Then
\[ \DM_\Gamma \preceq n^{\dim(G)}. \]
\end{pro}

\begin{proof}
Recall that we have fixed a closed embedding $G \rw \tx{SL}_m$ and are identifying $G$ with its image.
Let $\mc{X}$ be a finite set of generators for $\Gamma$ as a semigroup. Let $S$ be an integral basis for $\mc{O}_L$.
We claim that there exists $\lambda>0$ such that for any $A \in \Gamma$ with $\|A\|_{\mc{X}} = n$ and any non-zero coefficient $a \in \mc{O}_{L,f}$ of $A-I$ we have
\[ \|f^ka\|_S \leq \lambda^n \]
where $k$ is the least natural number such that $f^k a \in \mc{O}_L$.

The prove the claim, let $a'=a+1$ or $a'=a$ according to whether $a$ is a diagonal coefficient or not. Thus $a'$ is a coefficient of $A$. Let $K$ be the least natural number such that for all $X \in \mc{X}$, $f^KX \in \tx{M}_m(\mc{O}_L)$. Because $A$ is a product of exactly $n$ elements of $X$, we have $f^{nK}A \in \tx{M}_m(\mc{O}_L)$, and in particular $k<nK$. Then
\[ \|f^ka\|_S \leq \|f^{nK}a\|_S \leq \|f^{nK}a'\|_S + f^{nK}\|1\|_S. \]
This reduces the above claim to the following. There exists $\mu>0$ such that for any $A \in \Gamma$ with $\|A\|_\mc{X}=n$ and any non-zero coefficient $a \in \mc{O}_{L,f}$ of $A$ we have
\[ \|f^{nK}a\|_S \leq \mu^n. \]
We claim that if $\alpha$ denotes the maximum of $\|st\|_S$, where $s,t$ range over the elements of $S$, and $\beta$ denotes the maximum of $\|f^Kx\|_S$, where $x$ ranges over all entries of all elements of $\mc{X}$, then $\mu := m \alpha\beta$ satisfies the last statement. To see this, consider first the case $A=XY$ with $X,Y \in \mc{X}$. The entries of $A$ are scalar products of the rows of $X$ and the columns of $Y$. Thus we are led to study $\|x\cdot y\|_S$ for $x,y \in \mc{O}_L^m$, where $\cdot$ denotes scalar product. Clearly we have $\|x\cdot y\|_S \leq m\max\{\|x_iy_i\|_S:1 \leq i \leq m\}$. In terms of the basis $S$ we can write
\[ x_i = \sum_{s \in S} \lambda_{x,i,s}s\qquad \tx{and}\qquad y_i = \sum_{s \in S}\lambda_{y,i,s}s \]
where the $\lambda$'s belong to $\Z$. One computes
\[ \|x_iy_i\|_S \leq \|x_i\|_S\|y_i\|_S\max\{ \|st\|_S:\ s,t \in S \}. \]
This formula and induction on $n$ complete the proof of the claim.

To complete the proof of the proposition, let $A \in \Gamma$ be such that $\|A\|_\mc{X} \leq n$. Let $a$ be a non-zero entry of $A-I$ and $k$ the least integer with $f^ka \in \mc{O}_L$. According to Proposition \ref{OLCase} and the claim above there exists a natural number $M$, independent of $n$, and a homomorphism $\phi : \mc{O}_L \rw \F_p$ such that $p<Mn$ and $\phi(f^ka) \neq 0$. For all but finitely many primes $p$, we have that $\phi(f)$ is non-zero in $\F_p$. Hence, we may assume that $\phi$ extends to a homomorphism $\phi : \mc{O}_{L,f} \rw \F_p$ and $\phi(a) \neq 0$. The image of $A$ under the induced map $G(\mc{O}_{L,f}) \to G(\F_p)$ is non-trivial.
Further, according to Lemma \ref{lem:dimfin}, there exits $M' > 0$ such that $|G(\F_p)| \leq M' p^{\dim(G)}$. Hence, $|G(\F_p)| \leq M' (M n)^{\dim(G)}.$
\end{proof}

\begin{pro} \label{pro:killtrans}
Let $K \subset \C$ be a number field, $b \in \mc{O}_K \sm \{0\}$, and $\Gamma \subset G(\C)$ a finitely generated subgroup, such that $G(\mc{O}_{K,b}) \cap \Gamma$ is of finite-index in $\Gamma$. Then there exists a finite extension $L \subset \C$ of $K$, an element $f \in \Z \sm \{0\}$, and a homomorphism $\Gamma \rw G(\mc{O}_{L,f})$ with finite kernel.
\end{pro}

\begin{proof}
Let $S \subset \Gamma$ be a finite generating set. There exists a field $F \subset \C$, finitely generated over $K$, such that $S \subset G(F)$. Let $t_1,...,t_n$ be a transcendence basis for $F/K$. The extension $F/K(t_1,...,t_n)$ is finitely generated and algebraic, hence finite. Let $a \in F$ be a primitive element for that extension. Thus $F=K(t_1,...,t_n,a)$. The ring $\mc{O}_{K,b}[t_1,...,t_n]$ is a free polynomial algebra over $\mc{O}_{K,b}$ with field of fractions $K(t_1,...,t_n)$. There exists $s \in \mc{O}_K[t_1,...,t_n]$ such that the coefficients of the minimal polynomial of $a$ over $K(t_1,....,t_n)$ lie in the localization $\mc{O}_{K,b}[t_1,...,t_n]_s$. Thus the element $a$ is integral over $\mc{O}_{K,b}[t_1,...,t_n]_s$ and the ring $\mc{O}_{K,b}[t_1,...,t_n]_s[a] \subset F$ has $F$ as its field of fractions. Thus there exists $r \in \mc{O}_{K,b}[t_1,...,t_n]_s[a]$, such that if we put $R=\mc{O}_{K,b}[t_1,...,t_n]_s[a]_r$, then $S \subset G(R)$, and consequently $\Gamma \subset G(R)$.

We can find a homomorphism of $\mc{O}_{K,b}$-algebras
\[ \phi : \mc{O}_{K,b}[t_1,...,t_n] \rw \mc{O}_{K,b} \]
such that $\phi(s) \neq 0$. Then $\phi$ extends to a homomorphism
\[ \phi : \mc{O}_{K,b}[t_1,...,t_n]_s \rw \mc{O}_{K,b\phi(s)}. \]
There exists a finite extension $L \subset \C$ of $K$ such that the composition of $\phi$ with the natural inclusion $\mc{O}_{K,b\phi(s)} \rw K$ extends to a homomorphism
\[ \phi : \mc{O}_{K,b}[t_1,...,t_n]_s[a] \rw L. \]
The element $\phi(a) \in L$ is integral over $\mc{O}_{K,b\phi(s)}$, and hence belongs to $\mc{O}_{L,b\phi(s)}$. Thus in fact we obtain a homomorphism
\[ \phi : \mc{O}_{K,b}[t_1,...,t_n]_s[a] \rw \mc{O}_{L,b\phi(s)}. \]
We consider $\phi(r) \in \mc{O}_{L,b\phi(s)}$. Perturbing $\phi$ slightly if necessary, we may assume
%\footnote{
%$\phi(r)$ is a rational function in $\phi(t_1),...,\phi(t_n),\phi(a)$.
%Choose fixed values for $\phi(t_1),...,\phi(t_n),\phi(a)$ (so that $\phi$ is determined).
%Since $\phi(a)$ satisfies a polynomial whose coefficients depend analytically on $\phi(t_i)$,
%the implicit function theorem expresses $\phi(a)$ as an analytic function on $\phi(t_i)$ in a small neighborhood of the fixed values for them. Thus, all in all, $\phi(r)$ is an analytic function in $\phi(t_i)$ in a small neighborhood of the fixed values for them. This neighborhood contains non-zero values of that function.
%}
that $\phi(r) \neq 0$. In this way we obtain a homomorphism of $\mc{O}_K$-algebras
\[ \phi : R \rw \mc{O}_{L,b\phi(rs)}. \]
The algebra homomorphism $\mc{O}_L \otimes_\Z \Q \rw L$
given by multiplication is an isomorphism. Since $\Q = \varinjlim\limits_{f \in \Z} \Z_f$, we conclude that
\[ L \cong \varinjlim_{f \in \Z} \mc{O}_L \otimes_\Z \Z_f \cong \varinjlim_{f \in \Z} \mc{O}_{L,f}. \]
Thus there exists some $f \in \Z$ such that $[b\phi(rs)]^{-1} \in \mc{O}_{L,f}$. Composing $\phi$ with the inclusion $\mc{O}_{L,b\phi(rs)} \rw \mc{O}_{L,f}$ we finally arrive at a homomorphism of $\mc{O}_{K,b}$-algebras
\[ \phi : R \rw \mc{O}_{L,f}. \]
It induces a group homomorphism $\phi_* : G(R) \rw G(\mc{O}_{L,f})$ which fits into the commutative diagram

\begin{diagram}[small]
G(R)&&\rTo^{\phi_*}&&G(\mc{O}_{L,f})\\
&\luTo&&\ruTo\\
&&G(\mc{O}_{K,b})&&
\end{diagram}
The restriction of $\phi_*$ to $\Gamma$ is the desired homomorphism: Its kernel has trivial intersection with $G(\mc{O}_{K,b})$, i.e. it avoids a finite-index subgroup of $\Gamma$, and hence must be finite.
\end{proof}

\begin{cor} Let $\Gamma \subset G(\C)$ be a finitely generated subgroup. Assume that there exists a finite extension $K \subset \C$ of $\Q$ and $b \in \mc{O}_K \sm \{0\}$ such that $G(\mc{O}_{K,b}) \cap \Gamma$ is of finite-index in $\Gamma$. Then
\[ F_\Gamma(n) \preceq n^{\dim(G)}. \]
\end{cor}
\begin{proof}
This follows immediately from Proposition \ref{pro:killtrans}, Lemma \ref{DivisibilityAsymptoticLemma} and Proposition \ref{thecaseofO_L[1/f]}.
\end{proof}

%-----------------------------------------------------------

\noindent
Dept. of Mathematics, University of Chicago \\
5734 University Ave. Chicago, IL 60637 \\
E-mails: khalid@math.uchicago.edu, tkaletha@math.uchicago.edu
%-----------------------------------------------------------
%-----------------------------------------------------------

\end{document}